\documentclass[reqno,a4paper]{amsart}

\usepackage[
  style=numeric-comp,
  sorting=nty,
  sortcites=true,
  doi=false,
  url=false,
  giveninits=true
]{biblatex}
\renewbibmacro{in:}{}
\usepackage{amsmath}
\usepackage{amsthm}
\usepackage{amssymb}

\numberwithin{equation}{section}

\theoremstyle{plain}
\newtheorem{theorem}{Theorem}[section]
\newtheorem{lemma}[theorem]{Lemma}
\newtheorem{proposition}[theorem]{Proposition}
\newtheorem{corollary}[theorem]{Corollary}
\theoremstyle{remark}

\newtheorem*{remark*}{Remark}

\newcommand{\Cstar}{\mathbb{C}^{*}}

\title[Connection formulae for a generalised Ramanujan function]
  {Connection Formulae for a Generalised Ramanujan Entire Function}
\author{Joshua Holroyd}
\date{\today}

\begin{document}

\begin{abstract}
The Ramanujan--$q$-Airy connection formula relates the convergent
series at the origin to the behaviour at infinity of the Ramanujan
entire function. We extend this connection to a one-parameter
deformation, which embeds the Ramanujan (second-order) $q$-difference operator in a family of third-order equations. By contour integral as $q$-Borel inversion, we give behaviour at infinity in terms of divergent local expansions with connection coefficients uniquely determined. Discrete $q$-Borel--Laplace summation yields a convergent,
bilateral power series representation, whose dependence on summation path reflects the $q$-Stokes phenomenon. We prove remainder estimates establishing the formal, generally divergent, expansion at infinity as an asymptotic description of the entire function.
\end{abstract}

\maketitle

\section{Introduction}

Connection problems for linear $q$-difference equations relate solutions
normalised near fixed singular points of the equation, in particular,
the origin and infinity. The connection coefficients are generally
$q$-periodic functions~\cite{Morita2013}. When a local, formal series
solution is divergent, $q$-Borel--Laplace summation can produce a
meromorphic solution with that asymptotic expansion. In discrete
summation, one chooses a \emph{$q$-spiral} of points. Different choices can yield different solutions with the same
formal expansion in a given limit; this dependence is a $q$-analogue of the Stokes
phenomenon~\cite{Adachi2019}.

The Ramanujan entire function provides an explicit example of a
connection between the origin and infinity. Its defining power series around the origin
can be related to $q$-Airy functions that describe its behaviour at
infinity. Morita~\cite{Morita2014} used a $q$-Borel--Laplace procedure with
a contour integral as its inverse transform to derive this connection.
To give analytic meaning to a divergent local solution of the same
equation, they used a second procedure whose inverse transform is a
discrete sum along a $q$-spiral. These two types of $q$-Laplace
inversion, defined in Section~\ref{sec:preliminaries}, also enter the
present construction.

Related Ramanujan and $q$-Airy functions, their generalisations, and their asymptotics are studied
in~\cite{ElGuindyIsmail2016,IsmailZhang2007}.
Ohyama~\cite{Ohyama2011} places the Ramanujan and $q$-Airy functions
within a classification of $q$-special functions of Laplace type.

Fix $0<|q|<1$ and write $p=q^2$. We consider the deformation
\begin{equation}\label{eq:FConvergent}
  \begin{aligned}
    f(\lambda;b)
    &:=\sum_{j=0}^{\infty}
      \frac{p^{j^2}(-\lambda)^j}{(1-bp^j)(p;p)_j}=\frac{1}{1-b}
      {}_1\phi_2\left(b;bp,0;p,-p\lambda\right).
  \end{aligned}
\end{equation}
Throughout, $b\notin\{1/p^{m}:m\in\mathbb Z_{\geq0}\}$, unless
meromorphic continuation in $b$ is explicitly being discussed. The series converges
locally normally away from these parameter values, so $f$ is entire
in $\lambda$ and meromorphic in $b$, with at most simple poles at the
excluded values. At $b=0$, it reduces to the Ramanujan function
$A_p(\lambda)$. For general $b$, it satisfies the homogeneous
$p$-difference equation
\begin{equation}\label{eq:homogeneous-equation}
  \begin{aligned}
    f(\lambda;b)&-(1+b)f(p\lambda;b)+(b+p\lambda)f(p^2\lambda;b)-bp\lambda f(p^3\lambda;b)=0.
  \end{aligned}
\end{equation}
This equation is third order for $b\ne0$ and reduces to the Ramanujan equation at $b=0$, which is second order.

The defining series \eqref{eq:FConvergent} also satisfies the inhomogeneous first-order relation
\begin{equation}\label{eq:first-order-extension}
  f(\lambda;b)-bf(p\lambda;b)=A_p(\lambda).
\end{equation}
Write $(\sigma_p u)(\lambda)=u(p\lambda)$ and let
$R_p:=1-\sigma_p+p\lambda\sigma_p^2$ be the Ramanujan
operator, so that $R_pA_p=0$. Applying $R_p$ to
\eqref{eq:first-order-extension} gives $R_p(1-b\sigma_p)f=0$, which is precisely \eqref{eq:homogeneous-equation}.

For each allowed $b$, the first-order relation
\eqref{eq:first-order-extension} uniquely characterises $f$ among
functions holomorphic at $\lambda=0$. Equivalently, $f$ is the solution
of \eqref{eq:homogeneous-equation} holomorphic at the origin and
normalised by $f(0;b)=1/(1-b)$; see
Proposition~\ref{prop:governing-equations}. The condition at the origin
selects this particular function from the meromorphic solutions on
$\Cstar$.

\subsection{Main results}

Our focus is the particular entire function \eqref{eq:FConvergent},
its generally divergent expansion at infinity, and its representation
by convergent series. Proposition~\ref{prop:contour-coefficients}
constructs the formal connection expression by contour inversion, providing behaviour at infinity in terms of divergent local expansions, with connection coefficients uniquely determined. Theorem~\ref{thrm:MeromorphicResummation} constructs meromorphic solutions of Equation \eqref{eq:first-order-extension}, represented by convergent, bilateral power series expansions, dependent on $q$-Borel--Laplace summation path. Corollary~\ref{cor:weighted-asymptotics} proves remainder estimates that establish the divergent formal series as an asymptotic description of said meromorphic solutions, where, in particular, the entire function $f$ corresponds to a distinguished choice of summation path.

\subsection{Outline of the article}

Section~\ref{sec:preliminaries} fixes notation and summation conventions.
Section~\ref{sec:connection} develops the formal, divergent asymptotic series at
infinity. Section~\ref{sec:resummation} constructs meromorphic solutions by $q$-Borel--Laplace summation and establishes asymptotic results.
Section~\ref{sec:specializations} then examines dependence on the
$q$-Borel--Laplace summation spiral and the deformation parameter.
Appendix~\ref{app:discrete-summation} gives the proof of
Lemma~\ref{lem:discrete-summation}, including the summability and
remainder estimates that underpin Section~\ref{sec:resummation}.
Appendix~\ref{sec:continuum} presents a formal continuum limit and
its relation to Airy functions.

\section{Notation and summation conventions}\label{sec:preliminaries}

We fix the $q$-theta and $q$-Borel--Laplace conventions, then specify the domains used for the remainder estimates.

For $0<|q|<1$ and $n\in\mathbb Z_{\geq0}$, define the usual $q$-Pochhammer symbol
\begin{equation*}
  (a;q)_n:=\prod_{k=0}^{n-1}(1-aq^k),
  \qquad
  (a;q)_\infty:=\prod_{k=0}^{\infty}(1-aq^k).
\end{equation*}
An empty product is $1$, and
$(a_1,\ldots,a_m;q)_n:=\prod_{l=1}^m(a_l;q)_n$, with the same
convention for infinite products. The $q$-theta function is then defined
\begin{equation}\label{eq:theta-definition}
  \theta_q(z):=(z;q)_\infty(q/z;q)_\infty.
\end{equation}
This function is holomorphic on $\Cstar$, has simple zeros on
$q^{\mathbb Z}$, and satisfies $\theta_q(z)=\theta_q(q/z)$ and the
first-order $q$-difference equation $\theta_q(qz)=-\theta_q(z)/z$.
Iteration gives the identity
\begin{gather*}
  \theta_q(q^kz)=(-1/z)^kq^{-k(k-1)/2}\theta_q(z),\qquad\forall k\in\mathbb{Z}.
\end{gather*}
Furthermore, Jacobi's triple product gives the convergent bilateral expansion
\begin{equation*}
  (q;q)_\infty\theta_q(z)
  =\sum_{k\in\mathbb Z}q^{k(k-1)/2}(-z)^k.
\end{equation*}

For basic hypergeometric series, we use the standard
convention~\cite{GasperRahman2004}
\begin{equation*}
  {}_r\phi_s\left(
    \begin{matrix}a_1,\ldots,a_r\\b_1,\ldots,b_s\end{matrix};q,z\right)
  :=\sum_{n=0}^{\infty}
  \frac{(a_1,\ldots,a_r;q)_n}
       {(b_1,\ldots,b_s;q)_n(q;q)_n}
  \left((-1)^nq^{n(n-1)/2}\right)^{1+s-r}z^n.
\end{equation*}
In particular, the $q$-Airy function is
\begin{equation}\label{eq:q-airy-definition}
  \operatorname{Ai}_q(t)
  :={}_1\phi_1(0;-q;q,-t)
  =\sum_{n=0}^{\infty}
    \frac{q^{n(n-1)/2}t^n}{(p;p)_n},\qquad p=q^2.
\end{equation}
Its defining series is entire and satisfies the $q$-Airy equation
\begin{equation}\label{eq:q-airy-equation}
  \operatorname{Ai}_q(q^2t)+t\operatorname{Ai}_q(qt)
  -\operatorname{Ai}_q(t)=0.
\end{equation}

\subsection{Borel and Laplace transforms}

Let $0<|r|<1$. For a formal power series $h(t)=\sum_{n\geq0}a_nt^n$,
define the $r$-Borel transformations
\begin{equation*}
  \mathcal B_r^\pm h(\tau)
  :=\sum_{n=0}^{\infty}a_nr^{\pm n(n-1)/2}\tau^n.
\end{equation*}
With $\sigma_rh(t)=h(rt)$, the operational rule is
\begin{equation*}
  \mathcal B_r^\pm(t^m\sigma_r^nh)
  =r^{\pm m(m-1)/2}\tau^m
    \sigma_r^{n\pm m}\mathcal B_r^\pm h,
  \qquad m\geq0,\quad n\in\mathbb Z.
\end{equation*}
If $\phi$ is holomorphic near the origin, its minus-type $r$-Laplace transform is
\begin{equation*}
  \mathcal L_r^-\phi(t)
  :=\frac{(r;r)_\infty}{2\pi i}
    \int_{|\tau|=\varepsilon}
      \phi(\tau)\theta_r(-t/\tau)\frac{d\tau}{\tau},
\end{equation*}
where the circle is positively oriented and lies in the domain of
holomorphy. Whenever the transformed germ and integral are defined,
$\mathcal L_r^-\mathcal B_r^-h=h$.

For a meromorphic function $\phi$, a parameter $\sigma\in\Cstar$ is
\emph{admissible} if $\sigma r^{\mathbb Z}$ avoids its poles. Values at
removable singularities are taken after continuation in the Borel
variable. The discrete $r$-Laplace transform of level one is
\begin{equation}\label{eq:discrete-laplace-definition}
  \mathcal L_{r;1}^{[\sigma]}\phi(t)
  :=\frac{1}{(r;r)_\infty}
    \sum_{k\in\mathbb Z}
      \frac{\phi(r^k\sigma)}{\theta_r(-r^k\sigma/t)},
\end{equation}
on its domain of convergence, with possible poles in $t$ on
$-\sigma r^{\mathbb Z}$. This is the discrete summation convention
of~\cite{Zhang2002}. Reindexing the
sum gives invariance under $\sigma\mapsto r\sigma$. The triple product
also gives the moment identity
\begin{equation}\label{eq:laplace-moments}
  \mathcal L_{r;1}^{[\sigma]}(s^n)(t)
  =r^{-n(n-1)/2}t^n,\qquad n\geq0.
\end{equation}

We shall, in general, drop the $r$- prefix when discussing $r$-Borel and $r$-Laplace transforms, as this is the only context at hand.

\subsection{Asymptotic domains}\label{subsec:asymtptoticDomains}

We use closed domains of the form
$\{r^ku:u\in K,\ k\in\mathbb Z_{\geq0}\}$ near $t=0$, where $K$ is a
compact subset of an annulus in $\Cstar$ that avoids the relevant pole
spirals. Thus, these domains stay a fixed multiplicative distance from
the excluded spirals. We write
\begin{equation*}
  H(t)\sim_r\sum_{n=0}^{\infty}a_nt^n,
\end{equation*}
if, on each such domain, there are $A,C>0$ such that
\begin{equation*}
  \bigl|H(t)-\sum_{n=0}^{N-1}a_nt^n\bigr|
  \leq CA^N|r|^{-N(N-1)/2}|t|^N,
\end{equation*}
for every $N\geq0$ and all sufficiently small $t$. A function with this estimate for the zero series is called $r$-Gevrey flat. For solutions with theta prefactors, we explicitly state the corresponding weighted estimate; the prefactors cannot be omitted from the error bound.
\section{Formal connection by contour inversion}\label{sec:connection}

We determine the formal expansion associated with $f$ at
$\lambda=\infty$. The first-order relation \eqref{eq:first-order-extension} determines explicit connection data after $q$-Borel transformation. The corresponding contour integral inversion is applied coefficient-wise across a divergent local expansion, associating $f$ with a particular formal solution at infinity. The relevant asymptotic remainder estimate
for $f$ is established in Section~\ref{sec:resummation}.

\subsection{Normalisation and the formal Borel expansion}

The condition at the origin selects the function whose formal
connection expression we seek.

\begin{proposition}[Normalisation at the origin]
\label{prop:governing-equations}
For each allowed $b$, the defining series \eqref{eq:FConvergent} is
the unique solution of \eqref{eq:first-order-extension} holomorphic at
$\lambda=0$. Equivalently, it is the solution of
\eqref{eq:homogeneous-equation} holomorphic there and normalised by
$f(0;b)=1/(1-b)$.
\end{proposition}

\begin{proof}
We first determine the holomorphic solutions of the first-order
relation. Substituting $u(\lambda)=\sum_{j\geq0}a_j\lambda^j$ into
\eqref{eq:first-order-extension} gives
\[
  (1-bp^j)a_j=(-1)^jp^{j^2}/(p;p)_j,\qquad j\geq0.
\]
For every allowed $b$, each factor $1-bp^j$ is nonzero. These
relations, therefore, give precisely the coefficients of the entire
function \eqref{eq:FConvergent}, proving existence and uniqueness.
Applying the Ramanujan operator $R_p$ to
\eqref{eq:first-order-extension} also shows that $f$ satisfies
\eqref{eq:homogeneous-equation}.

Conversely, suppose that $u$ solves \eqref{eq:homogeneous-equation},
is holomorphic at the origin, and has $u(0)=1/(1-b)$.
Then $v(\lambda):=u(\lambda)-bu(p\lambda)$ satisfies
$R_pv=0$ and $v(0)=1$. In the Taylor recurrence for $R_pv=0$,
the coefficient of degree $j\geq1$ is determined by the preceding
coefficient because $1-p^j\ne0$. Thus, the constant term determines
the holomorphic solution uniquely, and $v=A_p$.
It follows that $u$ satisfies \eqref{eq:first-order-extension}, so
the uniqueness just proved gives $u=f$.
\end{proof}

Its minus-type $p$-Borel transform is the entire function
\begin{equation}\label{eq:BorelTransform}
  \phi_0(\mu):=(\mathcal B_p^-f)(\mu)
  =\sum_{j=0}^{\infty}
    \frac{p^{j(j+1)/2}(-\mu)^j}{(1-bp^j)(p;p)_j}.
\end{equation}
The first-order identity \eqref{eq:first-order-extension} for $f$ passes directly to $\phi_0$ and
determines its normalised formal expansion at infinity.

\begin{proposition}[Normalised formal Borel expansion]
\label{prop:formal-borel-expansion}
The function $\phi_0$ satisfies
\begin{equation}\label{eq:phi0-first-order}
  \phi_0(\mu)-b\phi_0(p\mu)=(p\mu;p)_\infty.
\end{equation}
Consequently, its theta normalisation satisfies
\begin{equation}\label{eq:normalized-first-order}
  Y(\mu)+b/(p\mu)Y(p\mu)
  =1/(1/\mu;p)_\infty,
  \qquad Y(\mu):=\phi_0(\mu)/\theta_p(p\mu).
\end{equation}
This equation has a unique formal solution
$A(\mu)=\sum_{k\geq0}c_k/\mu^{k}$, with $c_0=1$ and
\begin{equation}\label{eq:ck-first-order}
  c_k+(b/p^k)c_{k-1}=1/(p;p)_k,\qquad k\geq1.
\end{equation}
\end{proposition}

\begin{proof}
Subtracting $b\phi_0(p\mu)$ from the series for $\phi_0(\mu)$ in
\eqref{eq:BorelTransform} cancels the factor $1-bp^j$ in each
coefficient. Euler's product expansion
then gives
\[
  \phi_0(\mu)-b\phi_0(p\mu)
  =\sum_{j=0}^{\infty}\frac{p^{j(j-1)/2}(-p\mu)^j}{(p;p)_j}
  =(p\mu;p)_\infty,
\]
which is \eqref{eq:phi0-first-order}.
To obtain the equation for $Y$, divide by $\theta_p(p\mu)$ and use
\[
  \theta_p(p\mu)=(p\mu;p)_\infty(1/\mu;p)_\infty,
  \qquad \theta_p(p^2\mu)=-\theta_p(p\mu)/(p\mu).
\]
This yields \eqref{eq:normalized-first-order} as an identity of
meromorphic functions.

For the formal expansion at infinity, use
\[
  \frac{1}{(1/\mu;p)_\infty}
  =\sum_{k=0}^{\infty}\frac{1}{(p;p)_k}\mu^{-k},\qquad |\mu|>1.
\]
Substituting $A(\mu)=\sum_{k\geq0}c_k/\mu^k$ into
\eqref{eq:normalized-first-order} and comparing powers of $1/\mu$
gives $c_0=1$ and \eqref{eq:ck-first-order}. Each subsequent
coefficient is fixed by the preceding one, proving the existence and
uniqueness of the formal solution.
\end{proof}

Here, note that $c_0=1$ is not an independently chosen normalisation, rather it is forced by the inhomogeneous equation inherited from $f$. The formal Borel
expression to be inverted is $\theta_p(p\mu)A(\mu)$.

\subsection{Contour inversion and the formal connection expression}

We apply the same contour operator that recovers $f$ from $\phi_0$ to
finite truncations of its formal Borel expression. Introduce $t$ by
$\lambda=-q^3/t^2$, so that $\lambda=\infty$ corresponds to $t=0$, and
put
\[
  h(t):=\sum_{k=0}^{\infty}q^{k(k-1)/2}c_kt^k.
\]
This series is generally divergent. The following finite contour
identity determines the factors that accompany $h(t)$ and $h(-t)$ in an asymptotic description of $f$ as $\lambda\to\infty$.

\begin{proposition}[Formal connection by contour inversion]
\label{prop:contour-coefficients}
For $N\geq0$, define
\[
  \Phi_N(\mu):=\theta_p(p\mu)\sum_{k=0}^{N-1}c_k/\mu^{k}.
\]
For every $R>0$ and $\lambda=-q^3/t^2$,
\begin{align}
  \frac{(p;p)_\infty}{2\pi i}
  \int_{|\mu|=R}\Phi_N(\mu)\theta_p(-\lambda/\mu)\frac{d\mu}{\mu}
  ={}&\frac{\theta_q(t/q)}{2(-q;q)_\infty}
       \sum_{k=0}^{N-1}q^{k(k-1)/2}c_kt^k \nonumber\\
    &+\frac{\theta_q(-t/q)}{2(-q;q)_\infty}
       \sum_{k=0}^{N-1}q^{k(k-1)/2}c_k(-t)^k.
  \label{eq:finite-contour-connection}
\end{align}
Thus coefficientwise contour inversion associates $f(\lambda,b)$ with the formal expression
\begin{equation}\label{eq:FDivergent}
  \mathcal F(t;b):=\frac{\theta_q(t/q)h(t)
        +\theta_q(-t/q)h(-t)}{2(-q;q)_\infty}.
\end{equation}
Each theta-weighted term solves
\eqref{eq:homogeneous-equation} formally, with
$\lambda\mapsto p\lambda$ corresponding to $t\mapsto t/q$.
\end{proposition}

\begin{proof}
We compute the contour contribution of each formal coefficient, then check the difference equation for the resulting expression.
Since $\phi_0=\mathcal B_p^-f$ is entire by \eqref{eq:BorelTransform},
the inverse-transform identity in Section~\ref{sec:preliminaries}
gives
\[
  f(\lambda;b)=\frac{(p;p)_\infty}{2\pi i}
  \int_{|\mu|=R}\phi_0(\mu)\theta_p(-\lambda/\mu)\frac{d\mu}{\mu},
\]
for every $R>0$. Apply this contour operator to the finite expression
$\Phi_N$. With $s=1/\mu$, the reversal of orientation cancels the
sign in $d\mu/\mu=-ds/s$. Since $\theta_p(p/s)=\theta_p(s)$, the
contribution of $c_k$ is $(p;p)_\infty c_k I_k(\lambda)$, where
\[
  I_k(\lambda):=\frac{1}{2\pi i}
  \int_{|s|=1/R}\theta_p(s)\theta_p(-\lambda s)s^k\frac{ds}{s}.
\]
The two theta series converge absolutely and uniformly on this
circle, so the integral extracts the constant term of their product
with $s^k$. Jacobi's triple product gives
\begin{align*}
  (p;p)_\infty^2 I_k(\lambda)
  &=(-1)^kp^{k(k+1)/2}
    \sum_{m\in\mathbb Z}p^{m(m-1)}(-p^{k+1}\lambda)^m\\
  &=(p^2;p^2)_\infty(-1)^kp^{k(k+1)/2}
    \theta_{p^2}(p^{k+1}\lambda).
\end{align*}
To express this in terms of the theta factors in
\eqref{eq:finite-contour-connection}, separate the even powers in
the theta series as follows:
\[
  (q;q)_\infty\theta_q(x)+(q;q)_\infty\theta_q(-x)
  =2(q^4;q^4)_\infty\theta_{q^4}(-qx^2).
\]
With $x=q^{k+2}/t$, the argument $-qx^2$ becomes
$p^{k+1}\lambda$. The identity $\theta_q(z)=\theta_q(q/z)$ replaces
$\theta_q(\pm x)$ by $\theta_q(\pm t/q^{k+1})$; shifting these
arguments to $\pm t/q$ and using
$(p;p)_\infty=(q;q)_\infty(-q;q)_\infty$ gives
\[
  2(-q;q)_\infty(p;p)_\infty I_k(\lambda)
  =q^{k(k-1)/2}
       \left\{t^k\theta_q(t/q)+(-t)^k\theta_q(-t/q)\right\}.
\]
Multiplying by $c_k$ and summing for $0\leq k<N$ proves
\eqref{eq:finite-contour-connection}.

For the difference equation, multiply \eqref{eq:ck-first-order}
by $q^{k(k-1)/2}t^k$, sum over $k\geq1$, and use $c_0=1$ and
\eqref{eq:q-airy-definition}. This gives the formal identity
\begin{equation}\label{eq:h-first-order}
  h(t)+(bt/p)h(t/q)=\operatorname{Ai}_q(t).
\end{equation}
Since $\theta_q(t/q^2)=-(t/p)\theta_q(t/q)$, multiplication by
$\theta_q(t/q)$ yields
\[
  \theta_q(t/q)h(t)-b\theta_q(t/q^2)h(t/q)
  =\theta_q(t/q)\operatorname{Ai}_q(t).
\]
Under $\lambda=-q^3/t^2$, the shift $\sigma_p$ acts by $t\mapsto t/q$.
Applying $R_p$ to the right-hand side and dividing by $\theta_q(t/q)$ therefore gives
\begin{align*}
    \operatorname{Ai}_q(t)+(t/p)\operatorname{Ai}_q(t/q)
      -\operatorname{Ai}_q(t/q^2)=0,
\end{align*}
by \eqref{eq:q-airy-equation} at $t/q^2$.
Consequently $R_p(1-b\sigma_p)$ annihilates
$\theta_q(t/q)h(t)$, which is exactly
\eqref{eq:homogeneous-equation}. Replacing $t$ by $-t$ proves the
same statement for $\theta_q(-t/q)h(-t)$.
\end{proof}

Thus \eqref{eq:FDivergent} is completely determined by the
normalisation of $f$ at the origin. Corollary~\ref{cor:weighted-asymptotics}
will establish this formal expression as an asymptotic expansion of
$f$, with an explicit remainder estimate.

\section{Meromorphic connection formula}\label{sec:resummation}

Throughout this section, fix $b\in\mathbb C\setminus\{1/p^{m}:m\in\mathbb Z_{\geq0}\}$. We first apply discrete $q$-Borel--Laplace summation to the divergent power
series $h$, providing convergent, bilateral power series representations described asymptotically by $h$. Replacing $h(t)$ and $h(-t)$ in \eqref{eq:FDivergent} by
these sums gives meromorphic solutions of the equation for $f$.
We then prove a remainder estimate for both these sums and $f$,
measured relative to the theta factors in the formal expression.
Dependence on the summation spiral and the parameter $b$ is treated
in Section~\ref{sec:specializations}.

\subsection{Summation of the formal power series}

We first apply a discrete $q$-Borel--Laplace summation to the divergent, formal series $h$; yielding meromorphic functions that are proved to be asymptotic to $h$ as $t\to0$ as defined in Subsection \ref{subsec:asymtptoticDomains}.

\begin{lemma}[Borel transform and discrete summation]
\label{lem:discrete-summation}
For the fixed parameter $b$, define
\begin{equation}\label{eq:psi-product}
  \psi(s;b):=(-s;p)_\infty/(1+bs/p).
\end{equation}
At a removable singularity, this expression is continued in $s$, with
$b$ held fixed. Let $\sigma\in\Cstar$ be admissible for this continued
function. If $b\notin\{0,p,p^2,\ldots\}$, admissibility means
$\sigma\notin-q^{\mathbb Z}/b$; otherwise every $\sigma\ne0$ is
admissible. The positive Borel transform of $h$ is $\psi$, and its discrete Laplace sum is
\begin{equation}\label{eq:S-bilateral}
  \mathcal S_\sigma(t;b):=
  \frac{1}{(q;q)_\infty\theta_q(-\sigma/t)}
  \sum_{j\in\mathbb Z}
    \psi(q^j\sigma;b)q^{j(j-1)/2}(\sigma/t)^j.
\end{equation}
The bilateral series in \eqref{eq:S-bilateral} converges normally on compact subsets of $\Cstar$.
The resulting function is meromorphic in $t$, with possible poles on
$-\sigma q^{\mathbb Z}$, and satisfies
\begin{equation}\label{eq:S-first-order}
  \mathcal S_\sigma(t;b)+(bt/p)\mathcal S_\sigma(t/q;b)
  =\operatorname{Ai}_q(t).
\end{equation}
On closed asymptotic domains avoiding $-\sigma q^{\mathbb Z}$,
\begin{equation}\label{eq:S-asymptotic}
  \mathcal S_\sigma(t;b)\sim_q h(t)\qquad(t\to0).
\end{equation}

The formal series $h$ is entire when
$b\in\{0,p,p^2,\ldots\}$ and has radius of convergence zero for every
other allowed value of $b$. In the convergent cases,
$\mathcal S_\sigma$ equals its ordinary sum and is independent of
$\sigma$, with any apparent poles in $t$ removed. In particular, $\mathcal S_\sigma(t;0)=\operatorname{Ai}_q(t)$.
\end{lemma}

The proof, including the estimates for Laplace inversion, is given in Appendix~\ref{app:discrete-summation}.

\subsection{The exact connection formula}

Lemma \ref{lem:discrete-summation} (along with corresponding explicit asymptotic estimates in Appendix~\ref{app:discrete-summation}) serves as the backbone of the remaining arguments. We now show the construction of meromorphic solutions of Equation \eqref{eq:first-order-extension} explicitly in terms of $\mathcal{S}_\sigma$.

\begin{theorem}[Bilateral connection formula]
\label{thrm:MeromorphicResummation}
For fixed $b$ and an admissible $\sigma$, let
\begin{equation}\label{eq:F-symmetric}
  F_\sigma(\lambda;b):=
  \frac{\theta_q(t/q)\mathcal S_\sigma(t;b)
        +\theta_q(-t/q)\mathcal S_\sigma(-t;b)}{2(-q;q)_\infty},
  \qquad \lambda=-q^3/t^2.
\end{equation}
This defines a meromorphic function of $\lambda\in\Cstar$, with explicit
bilateral representation
\begin{align}
  2(p;p)_\infty F_\sigma(\lambda;b)={}&
  \frac{\theta_q(t/q)}{\theta_q(-\sigma/t)}
  \sum_{j\in\mathbb Z}
    \frac{(-q^j\sigma;p)_\infty}{1+bq^j\sigma/p}
    q^{j(j-1)/2}(\sigma/t)^j\nonumber\\
  &+\frac{\theta_q(-t/q)}{\theta_q(\sigma/t)}
  \sum_{j\in\mathbb Z}
    \frac{(-q^j\sigma;p)_\infty}{1+bq^j\sigma/p}
    q^{j(j-1)/2}(-\sigma/t)^j,
  \label{eq:convergent-connection}
\end{align}
where each quotient is evaluated as $\psi(q^j\sigma;b)$, with any common factor cancelled at the fixed parameter $b$, as in Lemma 4.1. It satisfies
\begin{equation}\label{eq:F-first-order}
  F_\sigma(\lambda;b)-bF_\sigma(p\lambda;b)=A_p(\lambda),
\end{equation}
and hence the homogeneous equation \eqref{eq:homogeneous-equation}.
\end{theorem}

\begin{proof}
Lemma~\ref{lem:discrete-summation} gives normal convergence of the
series defining $\mathcal S_\sigma$ and its meromorphicity in $t$.
The expression \eqref{eq:F-symmetric} is unchanged by $t\mapsto-t$,
so the two choices of $t$ for a given $\lambda=-q^3/t^2$ give the
same value. It therefore defines a meromorphic function of $\lambda$.
Substituting \eqref{eq:S-bilateral} into \eqref{eq:F-symmetric} and
using $(p;p)_\infty=(q;q)_\infty(-q;q)_\infty$ gives
\eqref{eq:convergent-connection}.

To prove \eqref{eq:F-first-order}, apply \eqref{eq:S-first-order}
at $t$ and $-t$ in \eqref{eq:F-symmetric}. The theta shifts
$\theta_q(\pm t/q^2)=\mp(t/p)\theta_q(\pm t/q)$ give
\begin{align*}
  F_\sigma(\lambda;b)-bF_\sigma(p\lambda;b)
  &=\frac{\theta_q(t/q)\operatorname{Ai}_q(t)
        +\theta_q(-t/q)\operatorname{Ai}_q(-t)}{2(-q;q)_\infty}=A_p(\lambda).
\end{align*}
The last equality is the connection formula
\eqref{eq:ramanujan-connection}, using
$(-1;q)_\infty=2(-q;q)_\infty$.
Applying $R_p$ to \eqref{eq:F-first-order} then gives
\eqref{eq:homogeneous-equation}.
\end{proof}

\begin{remark*}[Ramanujan specialisation]
In our theta normalisation, the Ramanujan--$q$-Airy
connection formula~\cite{Morita2014} is
\begin{equation}\label{eq:ramanujan-connection}
  (-1;q)_\infty A_{q^2}(-q^3/t^2)
  =\theta_q(t/q)\operatorname{Ai}_q(t)
   +\theta_q(-t/q)\operatorname{Ai}_q(-t),\qquad t\in\Cstar.
\end{equation}
At $b=0$, Lemma~\ref{lem:discrete-summation} gives
$\mathcal S_\sigma(t;0)=\operatorname{Ai}_q(t)$.
Thus \eqref{eq:F-symmetric} recovers this identity and gives
$F_\sigma(\lambda;0)=A_p(\lambda)$.
\end{remark*}

\subsection{The weighted asymptotic expansion of $f$}

We now establish the formal expression $\mathcal F$ as an asymptotic expansion of
the entire function $f$ as $\lambda\to\infty$. The estimate for $F_\sigma$ follows from
Lemma~\ref{lem:discrete-summation}; the first-order relation \eqref{eq:first-order-extension} controls
the difference $f-F_\sigma$.

Let $\mathcal F_N(t;b)$ denote the finite expression on the right-hand
side of \eqref{eq:finite-contour-connection}, retaining the terms
$0\leq k<N$, and let $W(t):=|\theta_q(t/q)|+|\theta_q(-t/q)|$. The two theta factors set the scale of the remainder.

\begin{corollary}[Asymptotics of the summed and entire solutions]
\label{cor:weighted-asymptotics}
Let $\sigma$ be admissible for the fixed parameter $b$. On every closed
asymptotic domain avoiding
$\sigma q^{\mathbb Z}\cup-\sigma q^{\mathbb Z}$, there are $A,C>0$
such that both $U=F_\sigma$ and $U=f$ satisfy
\begin{equation}\label{eq:weighted-asymptotic}
  \left|U(-q^3/t^2;b)-\mathcal F_N(t;b)\right|
  \leq CA^N|q|^{-N(N-1)/2}|t|^N W(t)
\end{equation}
for every $N\geq0$ and all sufficiently small $t$.
Thus, the formal contour expression \eqref{eq:FDivergent} is a
theta-weighted $q$-Gevrey expansion of the distinguished entire
solution.
\end{corollary}

\begin{proof}
Fix a closed asymptotic domain generated by a compact set $K$ as
in Section~\ref{sec:preliminaries}. Constants below may depend on
$q,b,\sigma,K$, but not on the point $t$ in this domain or the
truncation order $N$.
Apply \eqref{eq:S-asymptotic} on the domains generated by $K$ and
$-K$, taking common constants $A,C>0$ for the two remainder bounds.
Substitution into \eqref{eq:F-symmetric}, with its fixed normalising
factor absorbed into $C$, proves \eqref{eq:weighted-asymptotic}
for $F_\sigma$.

To obtain the estimate for $f$, we show that its difference from
$F_\sigma$ contributes no terms to the weighted expansion. Put
$D(t):=f(-q^3/t^2;b)-F_\sigma(-q^3/t^2;b)$.
Subtracting \eqref{eq:F-first-order} from
\eqref{eq:first-order-extension} gives $D(qt)=bD(t)$.
If $b=0$, then $D=0$ and there is nothing further to prove.
Assume $b\ne0$ and write $t=q^nu$, with $u\in K$ and $n\geq0$.

The two theta factors defining $W$ have no common zero, so $W$ is
bounded away from zero on $K$. Also, $D$ is holomorphic near $K$ by
\eqref{eq:F-symmetric} and the exclusions $\pm\sigma q^{\mathbb Z}$,
so it is bounded there.
Enlarge $C$ and choose $B\geq1$ so that
\[
  \sup_{u\in K}|D(u)|/W(u)\leq C,
  \qquad (|b|/|q|)\max_{u\in K}|u|\leq B.
\]
Iteration of $D(qt)=bD(t)$ and the theta shift identities then give
\begin{equation*}
  |D(q^nu)|/W(q^nu)
  =|bu/q|^n|q|^{n(n-1)/2}|D(u)|/W(u)
  \leq CB^n|q|^{n(n-1)/2}.
\end{equation*}

To obtain a bound uniform in the truncation order $N$, put
$r=\min_{u\in K}|u|>0$ and enlarge $A$ so that
$A r|q|\geq B$. Dividing the last bound by
$A^N|q|^{-N(N-1)/2}|q^nu|^N$ gives at most
\[
  C\left(B/(A r|q|)\right)^N
   B^{n-N}|q|^{(n-N)(n-N-1)/2}
  \leq C\sup_{k\in\mathbb Z}B^k|q|^{k(k-1)/2}<\infty.
\]
The supremum is finite because the quadratic exponent gives decay
in both directions. Absorbing it into $C$ proves the required
weighted bound for $D$, uniformly in $N,n,u$.
Adding this bound to the estimate for $F_\sigma$ proves
\eqref{eq:weighted-asymptotic} for $f$.
\end{proof}

\begin{remark*}[Scope of the asymptotic estimate]
Proposition~\ref{prop:distinguished-specialization} in the next section
identifies $f=F_{-p}$ for $b=0$ or $b\notin q^{\mathbb Z}$.
Directly specialising the estimate for $F_\sigma$ to this choice gives
the weighted expansion of $f$ for these parameters on closed
asymptotic domains avoiding $q^{\mathbb Z}\cup-q^{\mathbb Z}$.
Corollary~\ref{cor:weighted-asymptotics} establishes
\eqref{eq:weighted-asymptotic} for every allowed $b$ and on the
domains associated with any admissible $\sigma$. The first-order
relation controls $f-F_\sigma$ even when the two functions differ,
as they do at $b=p^m$, $m\geq1$; see
Proposition~\ref{prop:exceptional-specialization}.
\end{remark*}

\section{Dependence on $\sigma$ and $b$}\label{sec:specializations}

We now compare admissible $q$-Borel--Laplace summation
spirals at fixed $b$, then identify the distinguished sum that recovers
$f$ for generic $b$. Exceptional cases show that continuation in
$b$ must be distinguished from summation at a fixed parameter.

\subsection{Dependence on the summation spiral}

The first-order identity \eqref{eq:F-first-order} controls the
difference between sums formed on different spirals (varying $\sigma$).

\begin{proposition}[Dependence on the summation spiral]
\label{prop:stokes-difference}
Fix an allowed $b$, and let $\sigma$ and $\rho$ be admissible for $b$. Then
$\mathcal S_{q\sigma}=\mathcal S_\sigma$ and $F_{q\sigma}=F_\sigma$.
For $b\ne0$, the difference of two fixed-parameter sums has the form
\begin{equation}\label{eq:stokes-character}
  F_\sigma(\lambda;b)-F_\rho(\lambda;b)
  =H_{\sigma,\rho}(\lambda)
    \theta_p(\lambda)/\theta_p(\lambda/b),
  \qquad H_{\sigma,\rho}(p\lambda)=H_{\sigma,\rho}(\lambda),
\end{equation}
with $H_{\sigma,\rho}$ meromorphic. The difference
$\mathcal S_\sigma-\mathcal S_\rho$ is $q$-Gevrey flat on common
admissible domains. The difference $F_\sigma-F_\rho$ is flat relative
to $W(t)$.
\end{proposition}

\begin{proof}
Reindexing the discrete Laplace sum
\eqref{eq:discrete-laplace-definition} gives
$\mathcal S_{q\sigma}=\mathcal S_\sigma$; substitution into
\eqref{eq:F-symmetric} gives $F_{q\sigma}=F_\sigma$.

For $b\ne0$, define
$\Delta(\lambda):=F_\sigma(\lambda;b)-F_\rho(\lambda;b)$.
Subtracting \eqref{eq:F-first-order} for the two spirals gives
$\Delta(\lambda)=b\Delta(p\lambda)$. The theta shift identity gives
\[
  \theta_p(p\lambda)/\theta_p(p\lambda/b)
  =(1/b)\theta_p(\lambda)/\theta_p(\lambda/b).
\]
Thus
$H_{\sigma,\rho}(\lambda):=
\Delta(\lambda)\theta_p(\lambda/b)/\theta_p(\lambda)$
is meromorphic and $p$-periodic, proving \eqref{eq:stokes-character}.

On a closed asymptotic domain admissible for both spirals, take
common constants for the remainder bounds expressed by
\eqref{eq:S-asymptotic} and subtract the two estimates.
The common formal series cancels, proving $q$-Gevrey
flatness of $\mathcal S_\sigma-\mathcal S_\rho$.
Applying the same argument to \eqref{eq:weighted-asymptotic} proves
flatness of $F_\sigma-F_\rho$ relative to $W$. For $b=0$, both differences vanish by
Lemma~\ref{lem:discrete-summation} and \eqref{eq:F-first-order}.
\end{proof}

\subsection{The distinguished spiral and continuation in $b$} 

\begin{proposition}[The distinguished summation spiral]
\label{prop:distinguished-specialization}
For $b=0$ or $b\notin q^{\mathbb Z}$, the choice $\sigma=-p$ is
admissible and gives $F_{-p}(\lambda;b)=f(\lambda;b)$ for $\lambda\in\Cstar$.
\end{proposition}

\begin{proof}
Under the stated assumption on $b$, Lemma~\ref{lem:discrete-summation}
makes $\sigma=-p$ admissible. Substitute this choice into
\eqref{eq:convergent-connection}. The identities
$\theta_q(p/t)=\theta_q(t/q)$ and
$\theta_q(-p/t)=\theta_q(-t/q)$ make both theta quotients become $1$,
with removable values understood by continuation in $t$.
The terms with odd $j$ cancel between the two bilateral sums.
Writing $j=2k$ in the remaining terms and using
$\lambda=-q^3/t^2$ gives
\begin{equation}\label{eq:distinguished-bilateral}
  F_{-p}(\lambda;b)
  =\frac{1}{(p;p)_\infty}
    \sum_{k\in\mathbb Z}
    \frac{(pp^{k};p)_\infty}{1-bp^k}
    p^{k^2}(-\lambda)^k.
\end{equation}
The grouping is justified by the normal convergence in
Lemma~\ref{lem:discrete-summation}.
For $k<0$, the product $(pp^{k};p)_\infty$ contains a zero factor,
while $1-bp^k\ne0$ under the stated assumption on $b$.
These terms, therefore, vanish. For $k\geq0$, use
$(pp^{k};p)_\infty=(p;p)_\infty/(p;p)_k$.
The remaining sum is exactly \eqref{eq:FConvergent}.
\end{proof}

For fixed $\lambda$, the defining series for $f$ is meromorphic in
$b$. It therefore supplies a continuation of the generic family
$F_{-p}(\lambda;b)$, which we denote by $\hat{F}_{-p}(\lambda;b)$. Then we may write
\begin{equation}\label{eq:distinguished-continuation}
  \hat{F}_{-p}(\lambda;b)=f(\lambda;b)
\end{equation}
for every allowed $b$.
Here $F_\sigma$ denotes summation at a fixed $b$ as in Theorem \ref{thrm:MeromorphicResummation}; $\hat{F}_{-p}$ denotes
continuation in $b$ after summing with $\sigma=-p$ at generic $b$.
At $b=p^m$, $m\geq1$, these two operations give different values.

\subsection{Cancellation at the exceptional parameters}

At $b=p^m$, $m\geq1$, the Borel transform becomes the entire function
\[
  \psi_m(s)=\prod_{0\leq k\ne m-1}\bigl(1+sp^k\bigr).
\]
Lemma~\ref{lem:discrete-summation} shows that the resulting fixed-parameter sum is independent of $\sigma$.

\begin{proposition}[Exceptional specialisation]
\label{prop:exceptional-specialization}
For every $m\geq1$ and $\sigma\in\Cstar$,
\begin{equation}\label{eq:exceptional-correction}
  F_\sigma(\lambda;p^m)
  =f(\lambda;p^m)
   -p^{m(m+1)/2}(p;p)_{m-1}\lambda^{-m},
  \qquad\lambda\in\Cstar.
\end{equation}
\end{proposition}

\begin{proof}
At $b=p^m$, Lemma~\ref{lem:discrete-summation} shows that
$\mathcal S_\sigma$ is independent of $\sigma$; hence
\eqref{eq:F-symmetric} gives the same independence for $F_\sigma$.
We may therefore choose $\sigma=-p$.
Repeating the even-index grouping in the proof of
\eqref{eq:distinguished-bilateral}, now using the entire Borel
transform $\psi_m$, gives
\[
  F_{-p}(\lambda;p^m)
  =\frac{1}{(p;p)_\infty}\sum_{k\in\mathbb Z}
    \psi_m(-pp^{k})p^{k^2}(-\lambda)^k,
\]
where for $k\geq0$,
\[
  \psi_m(-pp^{k})
  =\frac{(pp^{k};p)_\infty}{1-p^{m+k}}
  =\frac{(p;p)_\infty}{(1-p^{m+k})(p;p)_k}.
\]
Thus, the nonnegative-index terms give $f(\lambda;p^m)$ by
\eqref{eq:FConvergent}. We need only compute the negative-index terms.

For $k<0$, the vanishing factor of $(-s;p)_\infty$ at
$s=-p^{k+1}$ has index $-k-1$. This factor is omitted from
$\psi_m$ only when $k=-m$; all other negative-index terms vanish.
At the surviving index, the product splits as
\[
  \psi_m(-p/p^m)
  =(p;p)_\infty\prod_{r=1}^{m-1}(1-1/p^{r})
  =-(-1)^{m}p^{-m(m-1)/2}(p;p)_{m-1}(p;p)_\infty.
\]
Adding this contribution to the nonnegative-index terms proves
\eqref{eq:exceptional-correction}.
\end{proof}

The difference between continuation in \(b\) and summation at fixed \(b=p^m\) can be seen directly in the \(k=-m\) term of the distinguished bilateral representation. The factor $(p/p^{m};p)_\infty/(1-b/p^{m})$ is zero for every \(b\ne p^m\), so its continuation in \(b\) is also zero at \(b=p^m\): removing the nonzero value \(\psi_m(-p/p^{m})\) whose contribution is the correction calculated in Proposition 5.3.

The correction is a constant multiple of \(1/\lambda^{m}\), satisfying the homogeneous part of \eqref{eq:first-order-extension} at \(b=p^m\); adding it preserves the first-order equation satisfied by \(f\) and \(F_\sigma\). Corollary 4.3 also shows that the correction contributes no terms to the asymptotic expansion measured relative to \(W(t)\). Consequently, \(f\) and \(F_\sigma\) have the same theta-weighted asymptotic expansion despite being different functions.

At these parameters, \(h\) converges and its ordinary sum equals \(\mathcal S_\sigma\). Substituting these convergent sums into the formal connection expression therefore gives \(F_\sigma\), which differs from \(f\) according to Proposition 5.3.

\section*{Summary}

We investigated a one-parameter deformation of the Ramanujan entire function, in particular, an extension of the Ramanujan--$q$-Airy connection formula.

A Borel transformation simplifies the formal connection problem; the corresponding Laplace contour inversion associates $f$ with a distinct combination of divergent series at infinity. Discrete $q$-Borel--Laplace summation produces a class of meromorphic solutions; explicit remainder estimates establish that these, and the entire function $f$, are indeed asymptotic to the divergent formal solution.
A distinguished summation spiral recovers $f$ for generic $b$, and continuation in $b$ extends this identity to every allowed parameter value.

As the summation spiral changes, the resulting meromorphic functions differ by solutions of the homogeneous part of the first-order equation \eqref{eq:first-order-extension}. These homogeneous solutions also
account for the correction at exceptional parameters $b=p^m$, $m\geq1$. In these cases, summation at
fixed $b$ and continuation from generic $b$ give different meromorphic functions despite
convergence of the formal series \(h\) near $\lambda=\infty$; these differences are invisible relative to the shared asymptotic expansion. At $b=0$, the
construction recovers the Ramanujan--$q$-Airy connection formula.

\printbibliography

\appendix
\section{Proof of Lemma~\ref{lem:discrete-summation}}\label{app:discrete-summation}

We retain the notation and parameter assumptions of
Lemma~\ref{lem:discrete-summation}. All sums are taken at a fixed $b$;
constants may depend on $q$, $b$, $\sigma$, and the indicated domain.

\begin{proof}
We first identify the Borel transform and establish convergence of
\eqref{eq:S-bilateral}. We then prove the first-order relation
\eqref{eq:S-first-order} and the asymptotic estimate
\eqref{eq:S-asymptotic}. The decomposition used for the latter also
determines when the formal series $h$ converges.

\emph{Borel transform and convergence.}
Multiply \eqref{eq:ck-first-order}, with $k=n$, by
$p^{n(n-1)/2}s^n$, sum over $n\geq1$, and use $c_0=1$. This gives
\begin{equation*}
  (1+bs/p)\sum_{n=0}^{\infty}p^{n(n-1)/2}c_ns^n
  =\sum_{n=0}^{\infty}\frac{1}{(p;p)_n}p^{n(n-1)/2}s^n
  =(-s;p)_\infty.
\end{equation*}
Since $p=q^2$, the power series multiplying $1+bs/p$ is
$\mathcal B_q^+h$.
Division by $1+bs/p$ identifies it with the germ of $\psi$ in
\eqref{eq:psi-product}.
Applying \eqref{eq:discrete-laplace-definition} and the theta shift
identity gives \eqref{eq:S-bilateral}.

We now prove normal convergence of the bilateral series in
\eqref{eq:S-bilateral}. Fix a compact set $T\subset\Cstar$.
As $j\to+\infty$, $\psi(q^j\sigma;b)\to1$, while
$q^{j(j-1)/2}(\sigma/t)^j$ decays at a Gaussian rate, uniformly
for $t\in T$. This controls the positive tail.

For the negative tail, write $j=-m$, with $m\geq0$. The product identity
\begin{equation*}
  (-a/p^M;p)_\infty
  =a^Mp^{-M(M+1)/2}(-p/a;p)_M(-a;p)_\infty,
  \qquad a\ne0,\quad M\geq0,
\end{equation*}
applied with base $|p|$ and $a=|\sigma|$ or $|\sigma/q|$, according
as $m=2M$ or $m=2M+1$, gives
\begin{equation*}
  \prod_{k=0}^{\infty}(1+|\sigma/q^m|\,|p|^k)
  \leq CB^m|q|^{-m^2/4},\qquad m\geq0.
\end{equation*}
Here $C,B>0$ depend only on $q$ and $\sigma$; factors exponential in
$m$ have been absorbed into $B^m$. This positive product bounds
$|(-\sigma/q^m;p)_\infty|$, including when any one factor is omitted.

For $b\notin\{0,p,p^2,\ldots\}$, the function $\psi(s;b)$ has a
single pole at $s=-p/b$, which the admissible sampling spiral
$\sigma q^{\mathbb Z}$ avoids. The sequence
$1/(1+bq^j\sigma/p)$ is therefore bounded: every term is finite,
and its limits as $j\to+\infty$ and $j\to-\infty$ are $1$ and $0$,
respectively. For $b=p^{k+1}$, $k\geq0$, the denominator instead
cancels the factor $1+sp^k$ in $(-s;p)_\infty$, so the continued
function $\psi(s;b)$ is the product with that factor omitted.
For $b=0$, it is the full product. Thus, for every allowed $b$,
after enlarging $C$ to depend also on $b$,
\begin{equation*}
  |\psi(\sigma/q^{m};b)|\leq CB^m|q|^{-m^2/4},\qquad m\geq0.
\end{equation*}
Multiplying by the remaining factors of the summand gives
\begin{equation*}
  \sup_{t\in T}
  |\psi(\sigma/q^{m};b)q^{m(m+1)/2}(t/\sigma)^m|
  \leq CB^m|q|^{m^2/4},\qquad m\geq0,
\end{equation*}
where $C,B$ may now depend on $T$ but remain independent of $m$.
This summable bound controls the negative tail and completes the
proof of the normal convergence of the bilateral series in
\eqref{eq:S-bilateral}. Dividing its sum by
$(q;q)_\infty\theta_q(-\sigma/t)$ gives the asserted meromorphic
function, with possible poles only on $-\sigma q^{\mathbb Z}$.

\emph{Bounds for Laplace inversion.}
We next establish bounds for termwise Laplace inversion and for the
remainders on the asymptotic domains of
Section~\ref{sec:preliminaries}. For $t\notin-\sigma q^{\mathbb Z}$, put
\begin{equation}\label{eq:absolute-kernel}
  \kappa_\sigma(t):=
  \sum_{j\in\mathbb Z}
  |(q;q)_\infty\theta_q(-q^j\sigma/t)|^{-1}.
\end{equation}
This is the sum of the absolute values of the Laplace weights. For
$\phi$ bounded on the sampling spiral, it gives
\begin{equation*}
  |\mathcal L_{q;1}^{[\sigma]}\phi(t)|
  \leq\kappa_\sigma(t)\sup_{j\in\mathbb Z}|\phi(q^j\sigma)|.
\end{equation*}
The theta shift identity gives normal convergence of the series in
\eqref{eq:absolute-kernel} on compact sets avoiding
$-\sigma q^{\mathbb Z}$. Reindexing gives
$\kappa_\sigma(q^l t)=\kappa_\sigma(t)$ for every $l\in\mathbb Z$.
Thus, a bound on the compact set $K$ from
Section~\ref{sec:preliminaries} extends to all its $q$-shifts:
\begin{equation*}
  \sup_{u\in K,l\in\mathbb Z}\kappa_\sigma(q^l u)
  =\sup_{u\in K}\kappa_\sigma(u)<\infty.
\end{equation*}

For $n\geq0$, the theta shift identity also gives
\begin{equation*}
  \frac{s^n}{\theta_q(-s/t)}
  =q^{-n(n-1)/2}t^n
    \frac{1}{\theta_q(-s/(t/q^n))}.
\end{equation*}
Taking absolute values, setting $s=q^j\sigma$, and summing over $j$
gives the absolute moment formula, using
$\kappa_\sigma(t/q^n)=\kappa_\sigma(t)$:
\begin{equation}\label{eq:absolute-moments}
  \sum_{j\in\mathbb Z}
  \frac{|q^j\sigma|^n}
       {|(q;q)_\infty\theta_q(-q^j\sigma/t)|}
  =|q|^{-n(n-1)/2}|t|^n\kappa_\sigma(t),\qquad n\geq0.
\end{equation}
Multiplying the same pointwise identity, with $n=N$, by
$\phi(s)/(q;q)_\infty$ and summing over $s=q^j\sigma$ gives
\begin{equation}\label{eq:laplace-multiplier}
  \mathcal L_{q;1}^{[\sigma]}(s^N\phi)(t)
  =q^{-N(N-1)/2}t^N
    \mathcal L_{q;1}^{[\sigma]}\phi(t/q^N),\qquad N\geq0,
\end{equation}
whenever these sums converge absolutely. We will use
\eqref{eq:absolute-moments} to justify termwise inversion and
\eqref{eq:laplace-multiplier} to invert the equation for $\psi$ and
estimate the pole contribution.

\emph{The first-order equation.}
By \eqref{eq:psi-product}, the Borel transform satisfies
$(1+bs/p)\psi(s;b)=(-s;p)_\infty$. To invert this relation, first
apply \eqref{eq:absolute-moments} to Euler's expansion of
$(-s;p)_\infty$. The absolute double sum is bounded by
\begin{equation*}
  \kappa_\sigma(t)\sum_{n=0}^{\infty}
  \frac{1}{|(p;p)_n|}|q|^{n(n-1)/2}|t|^n<\infty,
  \qquad t\notin-\sigma q^{\mathbb Z},
\end{equation*}
since the reciprocals of $(p;p)_n$ are bounded. Termwise inversion
is therefore justified, and \eqref{eq:laplace-moments} gives
\begin{align*}
  \mathcal L_{q;1}^{[\sigma]}[(-s;p)_\infty](t)
  &=\sum_{n=0}^{\infty}
    \frac{p^{n(n-1)/2}}{(p;p)_n}q^{-n(n-1)/2}t^n=\sum_{n=0}^{\infty}\frac{q^{n(n-1)/2}t^n}{(p;p)_n}
    =\operatorname{Ai}_q(t),
\end{align*}
using $p=q^2$ and \eqref{eq:q-airy-definition}.
The multiplier identity \eqref{eq:laplace-multiplier} with $N=1$
then yields
\begin{equation*}
  \mathcal L_{q;1}^{[\sigma]}[s\psi(s;b)](t)
  =t\,\mathcal S_\sigma(t/q;b).
\end{equation*}
The sum on the left converges absolutely because
\eqref{eq:S-bilateral} does at $t/q$. Applying the Laplace transform
to $\psi(s;b)+(b/p)s\psi(s;b)=(-s;p)_\infty$ now gives
\begin{equation*}
  \mathcal S_\sigma(t;b)+(bt/p)\mathcal S_\sigma(t/q;b)
  =\operatorname{Ai}_q(t).
\end{equation*}
This proves \eqref{eq:S-first-order} as an identity of meromorphic
functions.

\emph{The entire contribution.}
It remains to prove \eqref{eq:S-asymptotic} and determine when $h$
converges. For $b=0$, \eqref{eq:S-first-order} gives
$\mathcal S_\sigma(t;0)=\operatorname{Ai}_q(t)$, and
\eqref{eq:ck-first-order} identifies $h$ with its entire Taylor
series. Ordinary Taylor estimates give \eqref{eq:S-asymptotic} in
this case. Assume $b\ne0$ for the rest of the proof.

We separate $\psi$ into a simple-pole term and an entire remainder: we will invert the remainder term by term, then estimate the pole
term. Write
\begin{equation}\label{eq:psi-pole-decomposition}
  \begin{aligned}
    \psi(s;b)&=\frac{(p/b;p)_\infty}{1+bs/p}
                  +\psi_{\mathrm{ent}}(s;b),\qquad\psi_{\mathrm{ent}}(s;b):=
      \frac{(-s;p)_\infty-(p/b;p)_\infty}{1+bs/p},
  \end{aligned}
\end{equation}
so that $\psi_{\mathrm{ent}}$ is entire. For the fixed value of $b$, write
$\psi_{\mathrm{ent}}(s;b)=\sum_{n\geq0}d_ns^n$.
Expanding $(-s;p)_\infty$ by Euler's identity and using the definition \eqref{eq:psi-pole-decomposition} gives
\begin{equation*}
  d_n=-\sum_{k=0}^{\infty}
    \frac{p^{(n+1+k)(n+k)/2}(-p/b)^{k+1}}{(p;p)_{n+1+k}}.
\end{equation*}
The reciprocals of $(p;p)_j$ are uniformly bounded. Since
\begin{align*}
    \sum_{k\geq0}|p|^{k(k+1)/2}|p/b|^{k+1}<\infty,
\end{align*}
it follows that
$|d_n|\leq C|p|^{n(n+1)/2}$, where $C>0$ depends on the fixed
parameters $q,b$ but not on $n$. After the coefficient change for
Laplace inversion, this gives $|d_nq^{-n(n-1)/2}|\leq C|q|^{n(n+3)/2}$ so the resulting power series is entire. Moreover,
\eqref{eq:absolute-moments} bounds the absolute double sum obtained
by inserting the Taylor series of $\psi_{\mathrm{ent}}$ into its
Laplace transform by
\begin{equation*}
  \kappa_\sigma(t)\sum_{n=0}^{\infty}
       |d_n|\,|q|^{-n(n-1)/2}|t|^n<\infty,
  \qquad t\notin-\sigma q^{\mathbb Z}.
\end{equation*}
Interchanging sums and using \eqref{eq:laplace-moments} is therefore
justified and gives
\begin{equation}\label{eq:entire-laplace-sum}
  \mathcal L_{q;1}^{[\sigma]}\psi_{\mathrm{ent}}(t;b)
  =\sum_{n=0}^{\infty}d_nq^{-n(n-1)/2}t^n.
\end{equation}
Thus, this contribution is entire and independent of $\sigma$, with
any apparent poles of its Laplace representation removed.

Comparing Taylor coefficients in \eqref{eq:psi-pole-decomposition},
using $\mathcal B_q^+h=\psi$, gives
\begin{equation*}
  p^{n(n-1)/2}c_n=(p/b;p)_\infty(-b/p)^n+d_n,
  \qquad n\geq0.
\end{equation*}
Thus, termwise inversion of the decomposition recovers precisely $h$.

\emph{The pole contribution and the asymptotic estimate.}
If $(p/b;p)_\infty\ne0$, admissibility ensures that the sampled spiral
avoids $-p/b$. Write $P_\sigma$ for the sum of the simple-pole factor:
\begin{equation*}
  P_\sigma(t;b):=
  \mathcal L_{q;1}^{[\sigma]}(1+bs/p)^{-1}(t).
\end{equation*}
By linearity, the decomposition \eqref{eq:psi-pole-decomposition}
then gives
\begin{equation*}
  \mathcal S_\sigma(t;b)
  =(p/b;p)_\infty P_\sigma(t;b)
    +\sum_{n=0}^{\infty}d_nq^{-n(n-1)/2}t^n.
\end{equation*}
Expand $1/(1+bs/p)$ through degree $N-1$ and apply
\eqref{eq:laplace-multiplier} to the geometric remainder. This gives
\begin{equation}\label{eq:Euler-exact-remainder}
  \begin{aligned}
    P_\sigma(t;b)&-\sum_{n=0}^{N-1}(-b/p)^nq^{-n(n-1)/2}t^n=(-b/p)^Nq^{-N(N-1)/2}t^N P_\sigma(t/q^N;b).
  \end{aligned}
\end{equation}
The boundedness of $1/(1+bq^j\sigma/p)$ on the sampling spiral
and \eqref{eq:absolute-kernel} give
\begin{equation*}
  |P_\sigma(t/q^N;b)|\leq C\kappa_\sigma(t/q^N)
  =C\kappa_\sigma(t).
\end{equation*}
This is uniform in $N$ and in $t$ on each closed asymptotic domain
avoiding $-\sigma q^{\mathbb Z}$. The remainder in
\eqref{eq:Euler-exact-remainder} is therefore bounded by
$C|b/p|^N|q|^{-N(N-1)/2}|t|^N$.
The entire series in \eqref{eq:entire-laplace-sum} has an ordinary
Taylor remainder bound, which also satisfies the $q$-Gevrey estimate
since $|q|^{-N(N-1)/2}\geq1$. Combining the two bounds gives
\eqref{eq:S-asymptotic}, with formal series $h$ as identified by the
coefficient comparison above.

\emph{Convergence of $h$.}
If $(p/b;p)_\infty=0$, equivalently $b=p^{k+1}$ for some $k\geq0$,
the pole term in \eqref{eq:psi-pole-decomposition} vanishes. Both $h$
and its Laplace sum $\mathcal S_\sigma$ then equal the entire series
in \eqref{eq:entire-laplace-sum}. Thus, the sum is independent of
$\sigma$, its apparent poles in $t$ are removable, and ordinary
Taylor estimates give \eqref{eq:S-asymptotic}.

If $(p/b;p)_\infty\ne0$, the pole term contributes coefficients
\begin{align*}
    (p/b;p)_\infty(-b/p)^nq^{-n(n-1)/2},
\end{align*}
to $h$. The formal series with
these coefficients has a radius of convergence of zero. Adding the entire
series in \eqref{eq:entire-laplace-sum} cannot change that radius.
Together with the case $b=0$ treated earlier, this proves the stated
convergence classification and completes the proof.
\end{proof}

\section{A formal continuum limit}\label{sec:continuum}

The connection with Airy functions is also visible in a formal
differential limit of \eqref{eq:homogeneous-equation} near $b=1/2$.
In this scaling, solutions of the limiting equation are integrals of
exponentially weighted Airy functions.

We achieve this continuum limit by first setting $\lambda_n=p^n/4$ and
$f_n=f(\lambda_n;b)$. For $b\ne0$, write $f_n=g_n/b^{n}$.
Equation~\eqref{eq:homogeneous-equation} becomes
\[
  \begin{aligned}
    -bg_n&+(1+b)g_{n+1}-(1+p\lambda_n/b)g_{n+2}
      +(p\lambda_n/b)g_{n+3}=0.
  \end{aligned}
\]
Substituting $g_{n+k}=r^kg_n$
gives the characteristic polynomial
\[
  (r-1)(p\lambda_n(r/b)^2-r/b+1).
\]
At $p=1$ and $\lambda_n=1/4$, the quadratic factor has a double root
$r=2b$. All three roots therefore coalesce at $r=1$ when $b=1/2$.

To retain a parameter in the transition to this value, set
\[
  p=1-\varepsilon^3,\qquad
  b=1/(2+\beta\varepsilon),\qquad
  x=n\varepsilon,\qquad g_n=y_\varepsilon(x),
\]
where $\varepsilon\to0^+$, $\beta\in\mathbb C$ is fixed, and $x$
lies in a fixed compact set. The recurrence rearranges exactly as
\begin{align*}
  &g_{n+3}-3g_{n+2}+3g_{n+1}-g_n=\beta\varepsilon(g_{n+2}-g_{n+1})
    +(1-p\lambda_n/b^2)(g_{n+3}-g_{n+2}),
\end{align*}
where
\[
  1-p\lambda_n/b^2
  =-\beta\varepsilon+(x-\beta^2/4)\varepsilon^2
   +\mathcal{O}(\varepsilon^3).
\]
Assuming a smooth formal expansion
$y_\varepsilon(x)\sim y(x)+\varepsilon y_1(x)+\cdots$, with
corresponding expansions for the derivatives, comparison at order
$\varepsilon^3$ yields
\begin{equation*}
  y'''(x)+\beta y''(x)-(x-\beta^2/4)y'(x)=0.
\end{equation*}
The substitution $y'(x)=e^{-\beta x/2}v(x)$ reduces this equation to
$v''(x)=xv(x)$. Hence, its solutions have the form
\[
  y(x)=C_0+\int^x e^{-\beta s/2}
       \bigl\{C_1\operatorname{Ai}(s)+C_2\operatorname{Bi}(s)\bigr\}\,ds.
\]

For fixed nonzero $b\ne1/2$, the additional root remains separated
from the coalescing Airy pair. The parameter $\beta$ describes the
transition regime, with ordinary Airy integrals at $\beta=0$.
The reduction to an Airy equation for an exponentially rescaled
derivative mirrors the reduction
\eqref{eq:first-order-extension}.

\end{document}